\documentclass[twoside,11pt]{amsart}

\usepackage{indentfirst,latexsym,bm}
\usepackage{amsfonts,amssymb,amsmath,amsbsy}
\usepackage{dsfont,amsthm,hyperref,amscd}
\usepackage[all]{xy}
\usepackage{multirow}
\usepackage[titletoc]{appendix}
\allowdisplaybreaks[4]

\newtheorem{defi}{Definition}[section]
\newtheorem{thm}[defi]{Theorem}
\newtheorem{lem}[defi]{Lemma}
\newtheorem{pro}[defi]{Proposition}
\newtheorem{cor}[defi]{Corollary}
\newtheorem{rmk}[defi]{Remark}

\newcommand{\K}{\mathds{k}}
\newcommand{\Z}{\mathbb{Z}}

\newcommand{\F}{\mathbb{F}}
\newcommand{\G}{\mathbf{G}}

\newcommand{\Pp}{\mathcal{P}}
\newcommand{\cH}{\mathcal{H}}

\newcommand{\id}{\operatorname{id}}

\begin{document}
	
	\title[Bicrossed products with Radford algebras]{Bicrossed products of generalized Taft algebras and Radford algebras}
	\author{Rongchuan Xiong}
	\address{Department of Mathematics, Changzhou University, Changzhou 213164, China}
	\email{rcxiong@foxmail.com}
	
	\makeatletter
	\@namedef{subjclassname@2020}{\textup{2020} Mathematics Subject Classification}
	\makeatother
	
	\date{}
	
	\begin{abstract}
		Over an algebraically closed field of characteristic $p>0$, we classify
		all matched pairs between a generalized Taft algebra $T_{N,n,\xi}$ and
		the Radford algebra $R$. If $p\nmid N$, every matched pair is trivial,
		and the corresponding bicrossed product is the tensor product Hopf
		algebra. If $p\mid N$, matched pairs are parametrized by
		$\beta\in\mathbb F_p$, and the bicrossed products are described by
		explicit cross relations. Two such bicrossed products are isomorphic as
		Hopf algebras if and only if their parameters are equal; hence there are
		exactly $p$ isomorphism classes in this case.
		
	\end{abstract}
	
	\maketitle

	\section{Introduction}\label{sec:intro}
	
	The factorization problem for Hopf algebras asks which Hopf algebras factorize through two given Hopf algebras $A$ and $H$. Equivalently,
	one seeks to determine the matched pairs $(A,H,\triangleright,\triangleleft)$
	and to classify the associated bicrossed products $A\bowtie H$. By
	\cite[Proposition 3.12]{Majid90}, a Hopf algebra $E$ factorizes through $A$
	and $H$ if and only if $E$ is isomorphic to a bicrossed product
	$A\bowtie H$ arising from a matched pair. This construction is rooted in the
	group-theoretic theory of exact factorizations and was developed in the
	Hopf-algebraic setting by Takeuchi \cite{Takeuchi81} and Majid
	\cite{Majid90,Majid95}.
	
	A systematic computational approach to the factorization problem was
	developed by Agore, Bontea, and Militaru \cite{ABM14}. It has been applied to
	several families of Hopf algebras, including bicrossed products of two
	Sweedler algebras \cite{Bontea14}, two Taft algebras \cite{Agore18-Taft},
	Taft algebras with group Hopf algebras \cite{Agore-Nastasescu17}, $H_8$ with
	$H_4$ \cite{Lu-Ning-Wang20}, and generalized Taft algebras with group 
	algebras \cite{Wang-Cheng-Lu22}. Related crossed products were studied in
	\cite{Agore-Bontea-Militaru14,Agore18-Cn}. 
	
	The Taft algebra $T_{n^2}(\xi)$ was introduced by Taft \cite{Taft71} and is
	one of the standard examples of finite-dimensional noncommutative and
	noncocommutative Hopf algebras. When $n$ is prime, every noncommutative noncocommutative pointed Hopf algebra of dimension $n^2$ over an
	algebraically closed field of characteristic zero  is isomorphic to $T_{n^2}(\xi)$ \cite{AS98}. Related generalized Taft algebras, in which the group-like
	generator has order $N$ divisible by $n$, appear in the work of Radford
	\cite{Radford75} and were subsequently studied systematically by Huang,
	Chen, and Zhang \cite{HCZ04}. The same presentation defines a Hopf algebra in
	characteristic $p>0$ provided $(p,n)=1$; this is the setting considered here.
	See also \cite{CHYZ04} for the closely related class of monomial Hopf
	algebras.
	
	In  characteristic $p>0$, Radford \cite{Radford77} constructed a
	$p^2$-dimensional noncommutative and noncocommutative pointed Hopf algebra,
	now commonly called the Radford algebra.  Wang and Wang \cite{WW14} later classified pointed Hopf algebras of
	dimension $p^2$ over algebraically closed fields of characteristic $p$; the
	Radford algebra is the unique noncommutative and noncocommutative type in
	that classification.
	
	Compared with the characteristic-zero case, the factorization problem in
	positive characteristic has received less attention, and positive
	characteristic introduces new constraints. In this paper we classify all
	matched pairs between a generalized Taft algebra $T_{N,n,\xi}$ and the
	Radford algebra $R$ over an algebraically closed field $\K$ of
	characteristic $p>0$. Here $R$ is generated by a group-like element $g$ and a $(1,g)$-skew-primitive element
	$x$ subject to
	\[
	g^p=1,\qquad x^p=x,\qquad gx=xg+g(1-g),
	\]
	while
	$T_{N,n,\xi}$ is generated by a group-like element $G$ and a
	$(1,G)$-skew-primitive element $X$ subject to 
	\[
	G^N=1,\qquad X^n=0,\qquad GX=\xi XG,
	\]
	where $\xi$ is a primitive $n$-th root of unity and $n\mid N$.
	The relations $x^p=x$ and $G^N=1$ force the matched-pair parameter, say
	$\beta$, to lie in the prime field $\F_p$ and to satisfy $N\beta=0$;
	consequently, the classification splits into two cases according to whether
	$p$ divides $N$ or not. The main result is as follows.
	
	\begin{thm}\label{thm:main}
		Let $T=T_{N,n,\xi}$ and let $R$ be the Radford algebra.
		
		\medskip
		\noindent\textbf{Case 1: $p\nmid N$.}
		Every matched pair $(T,R,\triangleright,\triangleleft)$ is trivial:
		\[
		u\triangleright a=\varepsilon_R(u)a,
		\qquad
		u\triangleleft a=\varepsilon_T(a)u
		\qquad (u\in R,\ a\in T).
		\]
		Hence the bicrossed product is the tensor product Hopf algebra $T\otimes R$.
		
		\medskip
		\noindent\textbf{Case 2: $p\mid N$.}
		Matched pairs are parametrized by $\beta\in\F_p$. For the matched pair with
		parameter $\beta$, the actions on generators are
		\[
		g\triangleright G=G,\qquad g\triangleright X=X,
		\qquad x\triangleright G=0,\qquad x\triangleright X=\beta X,
		\]
		and
		\[
		g\triangleleft G=g,\qquad g\triangleleft X=0,
		\qquad x\triangleleft G=x+\beta(1-g),
		\qquad x\triangleleft X=0.
		\]
		The corresponding bicrossed product $T\bowtie_\beta R$ is generated by
		$G,X,g,x$ with the defining relations of $T$ and $R$ together with
		\[
		gG=Gg,\qquad gX=Xg,\qquad
		xG=Gx+\beta G(1-g),\qquad xX=Xx+\beta X.
		\]
		Moreover,
		\[
		T\bowtie_\beta R\cong T\bowtie_{\beta'}R
		\quad\text{as Hopf algebras}
		\quad\Longleftrightarrow\quad
		\beta=\beta'.
		\]
		Thus, when $p\mid N$, there are exactly $p$ isomorphism classes.
	\end{thm}
	
	The bicrossed products studied here also arise as a special case of
	liftings of the quantum-plane in \cite{Xpq2}; see
	Remark~\ref{rmk:quantum-plane}. The proof first determines the actions involving $g$ and $x$ by using the
	matched-pair identities on skew-primitive elements; this yields the
	parameter restrictions. The matched-pair axioms are then verified on PBW
	bases. Finally, the isomorphism classes are distinguished by the
	dimensions of skew-primitive spaces and the cross relations. 
	
	The paper is organized as follows. Section~\ref{sec:prelim} recalls the 
	necessary definitions and determines the relevant skew-primitive spaces. 
	Section~\ref{sec:first} determines the actions involving the group-like 
	generator $g$. Section~\ref{sec:second} determines the actions involving 
	$x$ and derives the parameter condition. Section~\ref{sec:exist} proves 
	sufficiency by checking all matched-pair axioms on PBW bases. 
	Section~\ref{sec:isom} proves the isomorphism classification.

	\section{Preliminaries}\label{sec:prelim}
	
	Throughout, $\K$ is an algebraically closed field of characteristic $p>0$.
	All vector spaces, algebras, coalgebras, and tensor products are taken over
	$\K$ unless otherwise stated. For a Hopf algebra $H$, we use the standard
	Sweedler notation
	\[
	\Delta(h)=h_{(1)}\otimes h_{(2)}
	\]
	for the comultiplication, and denote the counit and antipode by
	$\varepsilon$ and $S$, respectively. The set of group-like elements of $H$ is
	\[
	\G(H)=\{h\in H\mid h\ne0,\ \Delta(h)=h\otimes h\}.
	\]
	For group-like elements $u,v\in H$, the space of $(u,v)$-skew primitive
	elements is
	\[
	\mathcal P_{u,v}(H)=\{z\in H\mid \Delta(z)=z\otimes u+v\otimes z\}.
	\]
	In particular, $\Pp(H):=\Pp_{1,1}(H)$ is the space of primitive elements.
	
	Let $N\ge n\ge2$ with $n\mid N$ and $(p,n)=1$, and let $\xi\in\K$ be a
	primitive $n$-th root of unity. For an integer $m$ with $0\le m<n$, the
	quantum integer and quantum binomial coefficient are defined by
	\[
	[m]_{\xi}=\frac{\xi^m-1}{\xi-1},
	\qquad
	\binom{m}{k}_{\xi}
	=\frac{[m]_{\xi}[m-1]_{\xi}\cdots[m-k+1]_{\xi}}
	{[k]_{\xi}[k-1]_{\xi}\cdots[1]_{\xi}}
	\quad(0\le k\le m).
	\]
	Since $\xi$ has order $n$ and $m<n$, all quantum integers
	$[1]_{\xi},\dots,[m]_{\xi}$ are nonzero.

	\subsection{Matched pairs and bicrossed products}
	A matched pair of Hopf algebras $(A,H,\triangleright,\triangleleft)$ consists
	of linear maps
	\[
	\triangleright:H\otimes A\to A,
	\qquad
	\triangleleft:H\otimes A\to H
	\]
	such that $A$ is a left $H$-module coalgebra and $H$ is a right
	$A$-module coalgebra. Thus for all $u,v\in H$ and $a,b\in A$, 
	\[
	1_H\triangleright a=a,
	\qquad
	(uv)\triangleright a=u\triangleright(v\triangleright a),
	\]
	\[
	u\triangleleft1_A=u,
	\qquad
	(u\triangleleft a)\triangleleft b=u\triangleleft(ab),
	\]
	and
	\[
	\Delta_A(u\triangleright a)
	=\sum u_{(1)}\triangleright a_{(1)}
	\otimes u_{(2)}\triangleright a_{(2)},
	\qquad
	\varepsilon_A(u\triangleright a)
	=\varepsilon_H(u)\varepsilon_A(a),
	\]
	\[
	\Delta_H(u\triangleleft a)
	=\sum u_{(1)}\triangleleft a_{(1)}
	\otimes u_{(2)}\triangleleft a_{(2)},
	\qquad
	\varepsilon_H(u\triangleleft a)
	=\varepsilon_H(u)\varepsilon_A(a).
	\]
	The unit compatibilities are
	\[
	u\triangleright1_A=\varepsilon_H(u)1_A,
	\qquad
	1_H\triangleleft a=\varepsilon_A(a)1_H,
	\]
	and the matched-pair identities are, 
	\begin{align}
		u\triangleright(ab)
		&=\sum (u_{(1)}\triangleright a_{(1)})
		\bigl((u_{(2)}\triangleleft a_{(2)})\triangleright b\bigr),
		\tag{MP1}\label{eq:MP1}\\
		(uv)\triangleleft a
		&=\sum
		\bigl(u\triangleleft(v_{(1)}\triangleright a_{(1)})\bigr)
		\bigl(v_{(2)}\triangleleft a_{(2)}\bigr),
		\tag{MP2}\label{eq:MP2}\\
		\sum u_{(1)}\triangleleft a_{(1)}\otimes
		u_{(2)}\triangleright a_{(2)}
		&=\sum u_{(2)}\triangleleft a_{(2)}\otimes
		u_{(1)}\triangleright a_{(1)}.
		\tag{MP3}\label{eq:MP3}
	\end{align}
	
	The bicrossed product $A\bowtie H$ has underlying vector space $A\otimes H$,
	tensor-product coalgebra
	\[
	\Delta(a\bowtie u)
	=\sum(a_{(1)}\bowtie u_{(1)})\otimes
	(a_{(2)}\bowtie u_{(2)}),
	\]
	and multiplication
	\[
	(a\bowtie u)(b\bowtie v)
	=\sum a(u_{(1)}\triangleright b_{(1)})
	\bowtie (u_{(2)}\triangleleft b_{(2)})v.
	\]
	We shall suppress the symbol $\bowtie$ when no confusion can arise.

	\subsection{Generalized Taft algebras}
	
	The generalized Taft algebra $T=T_{N,n,\xi}$ is generated by $G,X$ with
	relations
	\[
	G^N=1,\qquad X^n=0,\qquad GX=\xi XG,
	\]
	and Hopf algebra structure
	\[
	\Delta(G)=G\otimes G,\qquad \varepsilon(G)=1,
	\qquad S(G)=G^{-1},
	\]
	\[
	\Delta(X)=X\otimes1+G\otimes X,\qquad
	\varepsilon(X)=0,\qquad S(X)=-G^{-1}X.
	\]
	The monomials
	\[
	\{G^iX^j\mid 0\le i<N,\ 0\le j<n\}
	\]
	form a basis; in particular, $\dim T=Nn$. The Hopf algebra is pointed, with
	group-like elements
	\[
	\G(T)=\{G^i\mid 0\le i<N\}.
	\]
	
	\begin{pro}\label{prop:primitive-T}
		For $1\le r<N$,
		\[
		\Pp_{1,G^r}(T)=
		\begin{cases}
			\K(G-1)\oplus\K X,&r=1,\\
			\K(G^r-1),&r\ne1.
		\end{cases}
		\]
	\end{pro}
	\begin{proof}
		For $0\le j<n$, the quantum binomial formula gives
		\begin{equation}\label{eq:taft-coproduct}
			\Delta(G^iX^j)
			=\sum_{s=0}^j
			\binom{j}{s}_{\xi}\xi^{-s(j-s)}
			G^{i+s}X^{j-s}\otimes G^iX^s.
		\end{equation}
		Because $j<n$ and $\xi$ has order $n$, every quantum integer
		$[m]_{\xi}$ with $1\le m\le j$ is nonzero. Hence
		$\binom{j}{s}_{\xi}\ne0$ for $0\le s\le j$.
		
		Let
		$z=\sum_{i,j}a_{i,j}G^iX^j\in\Pp_{1,G^r}(T)$ and suppose that $m\ge2$ is the
		largest $X$-degree occurring in $z$. In \eqref{eq:taft-coproduct}, the term
		with $s=1$ contributed by $G^iX^m$ is a nonzero multiple of
		\[
		G^{i+1}X^{m-1}\otimes G^iX.
		\]
		Both tensor factors have positive $X$-degree. Such a tensor does not occur in
		$z\otimes1+G^r\otimes z$. Moreover, among terms of total $X$-degree $m$, the tensor
		$G^{i+1}X^{m-1}\otimes G^iX$ can arise only from the $s=1$ term of
		$\Delta(G^iX^m)$; indeed, if it comes from $\Delta(G^{i'}X^m)$,
		then comparison of the second tensor factor forces $s=1$, and then
		comparison of the first tensor factor forces $i'=i$. Since the PBW tensors
		$G^aX^b\otimes G^cX^d$ are linearly independent, every $a_{i,m}$ must
		vanish, a contradiction. Descending on the maximal $X$-degree shows that
		$z$ has $X$-degree at most one.
		
		Write
		\[
		z=\sum_i a_iG^iX+\sum_i b_iG^i.
		\]
		Since
		\[
		\Delta(G^iX)=G^iX\otimes G^i+G^{i+1}\otimes G^iX,
		\]
		comparison of the terms with $X$ in the first tensor factor gives $a_i=0$
		for $i\ne0$. If $a_0\ne0$, comparison of the terms with $X$ in the second
		tensor factor gives $G=G^r$, hence $r=1$. Therefore an $X$-term occurs only
		for $r=1$, and then it is a scalar multiple of $X$.
		
		It remains to solve the equation inside the group algebra $\K\langle G\rangle$.
		A direct comparison of the tensors $G^i\otimes G^j$ gives
		\[
		\Pp_{1,G^r}(\K\langle G\rangle)=\K(G^r-1)
		\qquad (1\le r<N).
		\]
		This proves the stated formula.
	\end{proof}
	
	\begin{lem}\label{lem:P11-T}
		\[
		\Pp_{1,1}(T)=0
		\qquad\text{and}\qquad
		\Pp_{G,G}(T)=0.
		\]
	\end{lem}
	
	\begin{proof}
		Let $z=\sum_{i,j}a_{i,j}G^iX^j$ be primitive. The same maximal
		$X$-degree argument used in Proposition~\ref{prop:primitive-T} eliminates
		all terms of degree at least two. Thus
		\[
		z=\sum_i a_iG^iX+\sum_i b_iG^i.
		\]
		Comparing $G^iX\otimes G^i$ with the right-hand side
		$z\otimes1+1\otimes z$ gives $a_i=0$ for $i\ne0$. For $i=0$, the second
		summand of $\Delta(X)$ is $G\otimes X$, whereas the right-hand side contains
		$1\otimes a_0X$; hence $a_0=0$. We are left with an element of the group
		algebra $\K\langle G\rangle$, and comparison of the group-like tensor basis
		forces all $b_i$ to vanish. Therefore $\Pp_{1,1}(T)=0$.
		
		Finally, if $z\in\Pp_{G,G}(T)$, $G^{-1}z\in\Pp_{1,1}(T)=0$. Hence
		$z=0$, proving $\Pp_{G,G}(T)=0$.
	\end{proof}
	
	\subsection{The Radford algebra}
	
	The Radford algebra is
	\[
	R=\K\langle g,x\mid
	g^p=1,\ x^p=x,\ gx=xg+g(1-g)\rangle,
	\]
	with Hopf algebra structure
	\[
	\Delta(g)=g\otimes g,\qquad \varepsilon(g)=1,
	\qquad S(g)=g^{-1},
	\]
	\[
	\Delta(x)=x\otimes1+g\otimes x,\qquad
	\varepsilon(x)=0,\qquad S(x)=-g^{-1}x.
	\]
	The relation may also be written as
	\[
	gxg^{-1}=x+1-g.
	\]
	By the standard PBW reduction for the Radford algebra
	\cite{Radford77,WW14}, using $xg=gx-g+g^2$, the monomials
	\[
	\{g^ix^j\mid0\le i,j<p\}
	\]
	form a basis of $R$. Thus $\dim R=p^2$. The Hopf algebra
	is pointed, with
	\[
	\G(R)=\{g^i\mid0\le i<p\}.
	\]
	
	\begin{lem}\label{lem:primitive-R}
		For every $s\in\Z/p\Z$,
		\[
		\Pp_{1,g^s}(R)=
		\begin{cases}
			0,&s=0,\\
			\K(g-1)\oplus\K x,&s=1,\\
			\K(g^s-1),&s\ne0,1.
		\end{cases}
		\]
		Moreover,
		\[
		\Pp_{g^s,g^s}(R)=0
		\qquad\text{for every }s\in\Z/p\Z.
		\]
	\end{lem}
	\begin{proof}
		Give $g$ degree $0$ and $x$ degree $1$. The relation
		\[
		xg=gx+(g^2-g)
		\]
		shows that commuting $x$ past a power of $g$ preserves the leading
		$x$-degree and creates only lower-degree correction terms. Consequently, for
		$0\le m<p$,
		\begin{equation}\label{eq:radford-leading}
			\Delta(g^ix^m)
			=\sum_{k=0}^{m}\binom{m}{k}
			g^{i+k}x^{m-k}\otimes g^ix^k
			+\{\text{terms of total $x$-degree $<m$}\}.
		\end{equation}
		
		The extra terms have strictly smaller total $x$-degree and cannot cancel
		the leading tensors in \eqref{eq:radford-leading}.
		
		Let
		\[
		z=\sum_{i,j=0}^{p-1}a_{i,j}g^ix^j\in\Pp_{1,g^s}(R).
		\]
		If $m\ge2$ is the maximal $x$-degree occurring in $z$, then the $k=1$ term
		in \eqref{eq:radford-leading} contains
		\[
		m\,a_{i,m}\,g^{i+1}x^{m-1}\otimes g^ix.
		\]
		Since $1\le m\le p-1$, the scalar $m$ is nonzero in $\K$. The right-hand
		side $z\otimes1+g^s\otimes z$ contains no term in which both tensor factors
		have positive $x$-degree. In the associated graded tensor coalgebra, the
		bidegree-$(m-1,1)$ tensor
		$g^{i+1}x^{m-1}\otimes g^ix$ can arise only from the $k=1$ term attached to
		$g^ix^m$; lower filtered-degree corrections cannot contribute to this
		component. Linear independence of the PBW tensor basis therefore gives
		$a_{i,m}=0$ for every $i$, a contradiction. Descending on $m$ eliminates
		all terms of $x$-degree at least two. If $p=2$, there are no
		PBW terms with degree $2\le j<p$, so the conclusion follows directly.
		
		Hence
		\[
		z=\sum_i a_i g^ix+\sum_i b_i g^i.
		\]
		Using
		\[
		\Delta(g^ix)=g^ix\otimes g^i+g^{i+1}\otimes g^ix,
		\]
		comparison of terms with $x$ in the first tensor factor gives $a_i=0$ for
		$i\ne0$. If $a_0\ne0$, comparison of the terms with $x$ in the second tensor
		factor gives $g=g^s$, hence $s=1$. Therefore the only possible $x$-term is a
		multiple of $x$, and it occurs only when $s=1$.
		
		For the remaining group-algebra part, comparison in the basis
		$\{g^i\otimes g^j\}$ yields
		\[
		\Pp_{1,1}(\K\langle g\rangle)=0,
		\qquad
		\Pp_{1,g^s}(\K\langle g\rangle)=\K(g^s-1)
		\quad(s\ne0).
		\]
		This proves the displayed formula for $\Pp_{1,g^s}(R)$. In particular,
		\[
		\Pp_{1,1}(R)=0.
		\]
		Finally, if $z\in\Pp_{g^s,g^s}(R)$, then
		\[
		\Delta(g^{-s}z)=g^{-s}z\otimes1+1\otimes g^{-s}z,
		\]
		so $g^{-s}z\in\Pp_{1,1}(R)=0$. Hence $z=0$.
	\end{proof}

	\begin{cor}\label{cor:shifted-R}
		For $s\in\Z/p\Z$,
		\[
		\Pp_{g,g^s}(R)=g\Pp_{1,g^{s-1}}(R).
		\]
	\end{cor}
	
	\begin{proof}
		Multiplication by $g^{-1}$ gives a vector-space isomorphism
		$\Pp_{g,g^s}(R)\to\Pp_{1,g^{s-1}}(R)$.
	\end{proof}
	
	\section{The actions involving \texorpdfstring{$g$}{g}}\label{sec:first}
	
	Let $(T,R,\triangleright,\triangleleft)$ be a matched pair. We first determine
	the action involving the group-like generator $g\in R$.
	
	\begin{pro}\label{prop:first}
		In every matched pair $(T,R,\triangleright,\triangleleft)$,
		\[
		g\triangleright G=G,\qquad g\triangleright X=X,
		\qquad g\triangleleft G=g,\qquad g\triangleleft X=0.
		\]
		Consequently,
		\[
		g\triangleright a=a,
		\qquad
		g\triangleleft a=\varepsilon(a)g
		\qquad(a\in T).
		\]
	\end{pro}
	
	\begin{proof}
		Since $g\in\G(R)$ , the right module-coalgebra identity gives
		$g\triangleleft G\in \G(R)$; hence there exists $s\in\Z/p\Z$ such that
		\[
		g\triangleleft G=g^s.
		\]
		Since $X\in\Pp_{1,G}(T)$, we have
		\begin{equation}\label{eq:right-gX-coprod}
			\Delta(g\triangleleft X)
			=(g\triangleleft X)\otimes g
			+g^s\otimes(g\triangleleft X).
		\end{equation}
		Thus $g\triangleleft X\in\Pp_{g,g^s}(R)$.
		
		Let
		\[
		F:T\to T,\qquad F(a)=g\triangleright a.
		\]
		The left module property and $g^p=1$ give $F^p=\id_T$. Since the left action
		is a module coalgebra action, $F$ is a bijective coalgebra map. Therefore it
		permutes group-like elements. The unit condition gives $F(1)=1$, so
		\[
		F(G)=G^r
		\]
		for some $1\le r<N$. Applying the left module-coalgebra identity to $X$ gives
		\[
		F(X)\in\Pp_{1,G^r}(T).
		\]
		If $r\ne1$, Proposition~\ref{prop:primitive-T} implies
		$F(X)\in\K(G^r-1)\subseteq\K[\G(T)]$. A coalgebra automorphism permutes the
		group-like elements and therefore preserves their linear span
		$\K[\G(T)]$. Hence $F(X)\in\K[\G(T)]$ would imply
		$X\in\K[\G(T)]$, contradicting the PBW basis. Consequently,
		\[
		r=1.
		\]
		Thus
		\begin{equation}\label{eq:F-X}
			F(X)=\alpha(G-1)+\lambda X
		\end{equation}
		for some $\alpha,\lambda\in\K$. The same argument shows that $\lambda\ne0$;
		otherwise $F(X)$ would again lie in $\K[\G(T)]$.
		
		We now apply MP3 to $(u,a)=(g,X)$. Since
		$\Delta(g)=g\otimes g$ and $\Delta(X)=X\otimes1+G\otimes X$, we obtain
		\[
		(g\triangleleft X)\otimes1+g^s\otimes F(X)
		=g\otimes F(X)+(g\triangleleft X)\otimes G,
		\]
		or equivalently
		\begin{equation}\label{eq:MP3-gX}
			(g\triangleleft X)\otimes(1-G)
			=(g-g^s)\otimes F(X).
		\end{equation}
		Insert \eqref{eq:F-X}. The left-hand side of \eqref{eq:MP3-gX} has second
		tensor factor in $\K(1-G)$, whereas the right-hand side has $X$-component
		\[
		\lambda(g-g^s)\otimes X.
		\]
		Since $\lambda\ne0$ and $X$ is linearly independent from $1-G$, it follows
		that $g=g^s$ and so $s=1$. Equation \eqref{eq:MP3-gX} then becomes
		\[
		(g\triangleleft X)\otimes(1-G)=0,
		\]
		and $G\ne1$ gives
		\[
		g\triangleleft X=0.
		\]
		
		The right module identity determines the action of $g$ on the PBW basis:
		\[
		g\triangleleft G^i=g,
		\qquad
		g\triangleleft G^iX^j=0\quad(j>0).
		\]
		Therefore
		\begin{equation}\label{eq:g-right-trivial}
			g\triangleleft a=\varepsilon(a)g
			\qquad(a\in T).
		\end{equation}
		
		Now MP1 implies that $F=g\triangleright-$ is an algebra map. Indeed, using
		\eqref{eq:g-right-trivial},
		\[
		\begin{aligned}
			F(ab)
			&=\sum F(a_{(1)})
			\bigl((g\triangleleft a_{(2)})\triangleright b\bigr)\\
			&=\sum F(a_{(1)})\varepsilon(a_{(2)})F(b)
			=F(a)F(b).
		\end{aligned}
		\]
		Thus $F$ is a bialgebra automorphism. Since $F(G)=G$, iterating
		\eqref{eq:F-X} gives
		\[
		F^k(X)
		=\alpha(1+\lambda+\cdots+\lambda^{k-1})(G-1)+\lambda^kX.
		\]
		The equality $F^p=\id$ yields $\lambda^p=1$. In characteristic $p$,
		\[
		t^p-1=(t-1)^p,
		\]
		so the unique solution is $\lambda=1$. Hence
		\[
		F(X)=\alpha(G-1)+X.
		\]
		Finally, applying $F$ to the relation $GX=\xi XG$ yields
		\[
		\alpha(1-\xi)G(G-1)=0.
		\]
		Since $\xi\ne1$ and $G(G-1)\ne0$, we obtain $\alpha=0$. Therefore
		$F=\id_T$, completing the proof.
	\end{proof}
	
	\section{The actions involving \texorpdfstring{$x$}{x}}\label{sec:second}
	
	We now determine the action involving the $(1,g)$-skew-primitive generator $x$.
	
	From the left module-coalgebra identity and Proposition~\ref{prop:first},
	\[
	x\triangleright G\in\Pp_{G,G}(T)=0,
	\qquad
	x\triangleright X\in\Pp_{1,G}(T)
	=\K(G-1)\oplus\K X.
	\]
	Thus
	\begin{equation}\label{eq:x-left-initial}
		x\triangleright G=0,
		\qquad
		x\triangleright X=\alpha(G-1)+\beta X.
	\end{equation}
	Similarly, the right module-coalgebra identity gives
	\[
	x\triangleleft G\in\Pp_{1,g}(R),
	\qquad
	x\triangleleft X\in\Pp_{1,g}(R),
	\]
	so we may write
	\begin{equation}\label{eq:x-right-initial}
		x\triangleleft G=a(g-1)+bx,
		\qquad
		x\triangleleft X=c(g-1)+dx.
	\end{equation}
	
	\subsection{The MP3 comparison}
	
	Apply MP3 to $(u,a)=(x,X)$. Using Proposition~\ref{prop:first} and
	\eqref{eq:x-left-initial}--\eqref{eq:x-right-initial}, the two sides are
	\[
	\begin{aligned}
		\mathrm{LHS}
		&=(x\triangleleft X)\otimes1
		+(x\triangleleft G)\otimes X
		+g\otimes\bigl(\alpha(G-1)+\beta X\bigr),\\
		\mathrm{RHS}
		&=1\otimes\bigl(\alpha(G-1)+\beta X\bigr)
		+x\otimes X+(x\triangleleft X)\otimes G.
	\end{aligned}
	\]
	Put $q=g-1$ and $Q=G-1$. After cancelling the common terms, the identity
	above is equivalent to
	\[
	(aq+bx)\otimes X+q\otimes(\alpha Q+\beta X)
	=x\otimes X+(cq+dx)\otimes Q.
	\]
	The four tensors
	$q\otimes Q$, $q\otimes X$, $x\otimes Q$, and $x\otimes X$ are linearly
	independent. Comparing their coefficients gives
	\[
	c=\alpha,
	\qquad a=-\beta,
	\qquad b=1,
	\qquad d=0.
	\]
	Consequently,
	\begin{equation}\label{eq:x-right-mid}
		x\triangleleft G=x+\beta(1-g),
		\qquad
		x\triangleleft X=\alpha(g-1).
	\end{equation}
	
	\subsection{The MP1 comparison}
	
	Apply MP1 first to $(u,a,b)=(x,G,X)$. Because
	$x\triangleright G=0$ and $(1-g)\triangleright X=0$, we get
	\[
	x\triangleright(GX)
	=G\bigl((x\triangleleft G)\triangleright X\bigr)
	=G(x\triangleright X)
	=\alpha(G^2-G)+\beta GX.
	\]
	Next apply MP1 to $(u,a,b)=(x,X,G)$. The only nonzero contribution is
	\[
	x\triangleright(XG)
	=(x\triangleright X)G
	=\alpha(G^2-G)+\beta XG,
	\]
	because $x\triangleright G=0$ and
	$(x\triangleleft X)\triangleright G=\alpha(g-1)\triangleright G=0$.
	Since $GX=\xi XG$ and the action is linear,
	\[
	x\triangleright(GX)=\xi\,x\triangleright(XG).
	\]
	Thus
	\[
	\alpha(1-\xi)(G^2-G)=0.
	\]
	Since $\xi\ne1$, it follows that $\alpha=0$. Therefore
	\begin{equation}\label{eq:x-actions-generators}
		x\triangleright G=0,
		\qquad x\triangleright X=\beta X,
		\qquad x\triangleleft G=x+\beta(1-g),
		\qquad x\triangleleft X=0.
	\end{equation}
	
	\subsection{The parameter restrictions}
	
	The left module identity and $x^p=x$ imply that the operator
	$D=x\triangleright-$ satisfies $D^p=D$. Applying this equality to $X$ and
	using \eqref{eq:x-actions-generators} gives
	\[
	\beta^pX=\beta X.
	\]
	Hence
	\begin{equation}\label{eq:beta-Fp}
		\beta^p=\beta,
		\qquad\text{so }\beta\in\F_p.
	\end{equation}
	
	Let
	\[
	\rho_G(u)=u\triangleleft G.
	\]
	Then $G^N=1$ gives $\rho_G^N=\id_R$. By
	\eqref{eq:x-actions-generators},
	\[
	\rho_G(g)=g,
	\qquad
	\rho_G(x)=x+\beta(1-g).
	\]
	Linearity therefore gives, for every $k\ge0$,
	\[
	\rho_G^k(x)=x+k\beta(1-g).
	\]
	Taking $k=N$ yields
	\[
	N\beta(1-g)=0.
	\]
	Since $g\ne1$,
	\begin{equation}\label{eq:beta-N}
		N\beta=0.
	\end{equation}
	Thus $\beta=0$ when $p\nmid N$, whereas every $\beta\in\F_p$ satisfies
	\eqref{eq:beta-N} when $p\mid N$.
	
	\subsection{The actions on PBW bases}
	
	Since $g\triangleright G=G$ and $x\triangleright G=0$, it follows that for
	$v\in R$,
	\begin{equation}\label{eq:R-on-G-counit}
		v\triangleright G=\varepsilon(v)G.
	\end{equation}
	Apply MP2 with $a=G$. Using \eqref{eq:R-on-G-counit},
	\[
	\begin{aligned}
		(uv)\triangleleft G
		&=\sum
		\bigl(u\triangleleft(\varepsilon(v_{(1)})G)\bigr)
		(v_{(2)}\triangleleft G)\\
		&=(u\triangleleft G)(v\triangleleft G).
	\end{aligned}
	\]
	Thus the linear operator
	\[
	\rho_G:R\to R,\qquad \rho_G(u)=u\triangleleft G,
	\]
	is an algebra map. The right module identity and $G^N=1$ give
	$\rho_G^N=\id_R$, so $\rho_G$ is an algebra automorphism. Since its values on
	the algebra generators are
	\[
	\rho_G(g)=g,
	\qquad
	\rho_G(x)=x+\beta(1-g),
	\]
	$\rho_G$ is determined by $\beta$. Write $\phi_\beta:=\rho_G$. Then
	\begin{equation}\label{eq:rhoG-phi}
		\phi_\beta(g)=g,
		\qquad
		\phi_\beta(x)=x+\beta(1-g).
	\end{equation}
	
	It remains to determine the operator
	$\rho_X(u)=u\triangleleft X$ on all of $R$. The left module identity gives
	\[
	v\triangleright X=\chi_1(v)X,
	\]
	where $\chi_1(g)=1$ and $\chi_1(x)=\beta$. Applying MP2 to $a=X$ and using
	$\Delta(X)=X\otimes1+G\otimes X$ gives the recursive identity
	\begin{equation}\label{eq:rhoX-recursion}
		\rho_X(uv)
		=\rho_X(u)L_{\chi_1}(v)+\rho_G(u)\rho_X(v),
	\end{equation}
	where
	\[
	L_{\chi_1}(v)=\sum\chi_1(v_{(1)})v_{(2)}.
	\]
	By \eqref{eq:x-actions-generators} and Proposition~\ref{prop:first},
	\[
	\rho_X(1)=\rho_X(g)=\rho_X(x)=0.
	\]
	Since $R$ is generated as an algebra by $g$ and $x$,
	\eqref{eq:rhoX-recursion} implies inductively on word length that
	\begin{equation}\label{eq:rhoX-zero}
		\rho_X=0
		\quad\text{on all of }R.
	\end{equation}
	The right module identity now yields the full right action
	\begin{equation}\label{eq:right-basis-necessity}
		u\triangleleft G^iX^j=
		\begin{cases}
			\phi_\beta^i(u),&j=0,\\
			0,&j>0.
		\end{cases}
	\end{equation}
	
	MP1 now shows that $D=x\triangleright-$ is a derivation of $T$. By
	linearity, it suffices to take $a=G^iX^j$. Writing
	$\Delta(x)=x\otimes1+g\otimes x$, the first term in MP1 contributes
	$(x\triangleright a)b$. In the second term, the factor
	$x\triangleleft a_{(2)}$ is nonzero only for the summand of $\Delta(a)$ whose
	second tensor factor has $X$-degree zero, namely the $s=0$ term. For that
	summand, $a_{(1)}=a$, $a_{(2)}=G^i$, and
	$x\triangleleft G^i=x+i\beta(1-g)$. Since
	$(1-g)\triangleright b=0$ by Proposition~\ref{prop:first}, this term is
	$a(x\triangleright b)$. Hence
	\[
	x\triangleright(ab)
	=(x\triangleright a)b+a(x\triangleright b).
	\]
	Together with $D(G)=0$ and $D(X)=\beta X$, this gives
	\begin{equation}\label{eq:left-basis-necessity}
		x\triangleright(G^iX^j)=j\beta G^iX^j.
	\end{equation}
	The left module identity then determines the action of every $u\in R$.
	
	We summarize the necessity part.
	
	\begin{pro}\label{prop:second}
		In every matched pair $(T,R,\triangleright,\triangleleft)$ there exists a
		unique $\beta\in\F_p$ satisfying $N\beta=0$ such that
		\eqref{eq:x-actions-generators} holds. Equivalently, the full actions are the
		ones described by \eqref{eq:right-basis-necessity} and
		\eqref{eq:left-basis-necessity}, together with
		$g\triangleright a=a$ and $g\triangleleft a=\varepsilon(a)g$.
	\end{pro}
	
	\begin{proof}
		Existence of $\beta$ and all displayed formulas were proved above. Uniqueness
		follows from $x\triangleright X=\beta X$ and $X\ne0$.
	\end{proof}
	
	\section{Construction of matched pairs}\label{sec:exist}
	
	We now prove that every parameter obtained in Proposition~\ref{prop:second}
	actually yields a matched pair.
	
	Fix $\beta\in\F_p$ with $N\beta=0$. Let
	\[
	D(G)=0,
	\qquad
	D(X)=\beta X.
	\]
	The assignment extends to a derivation of $T$: it sends the defining
	relations $G^N-1$, $X^n$, and $GX-\xi XG$ into the ideal that they generate.
	Explicitly,
	\begin{equation}\label{eq:D-basis}
		D(G^iX^j)=j\beta G^iX^j.
	\end{equation}
	For every integer $j$, let
	\begin{equation}\label{eq:chi-j}
		\chi_j:R\to\K,
		\qquad
		\chi_j(g)=1,
		\qquad
		\chi_j(x)=j\beta.
	\end{equation}
	Because $j\beta\in\F_p$, the defining relations of $R$ show that $\chi_j$ is
	an algebra character.
	
	Define the left action by
	\begin{equation}\label{eq:left-action-existence}
		u\triangleright G^iX^j
		=\chi_j(u)G^iX^j.
	\end{equation}
	Equivalently, $g$ acts as the identity and $x$ acts as $D$. Since
	$D^p=D$ by \eqref{eq:D-basis}, the operator relations corresponding to the
	relations of $R$ are
	\[
	\id_T^p=\id_T,
	\qquad D^p=D,
	\qquad \id_TD=D\id_T+\id_T(\id_T-\id_T).
	\]
	Thus \eqref{eq:left-action-existence} gives a well-defined left
	$R$-module structure on $T$.
	
	Next define
	\[
	\phi_\beta(g)=g,
	\qquad
	\phi_\beta(x)=x+\beta(1-g).
	\]
	If $m\in\{0,\dots,p-1\}$ represents $\beta\in\F_p$, then the relation
	$gxg^{-1}=x+1-g$ gives
	\[
	\phi_\beta=\operatorname{Ad}_{g^m}.
	\]
	Thus $\phi_\beta$ is a Hopf automorphism of $R$, and
	\[
	\phi_\beta^k(x)=x+k\beta(1-g),
	\qquad
	\phi_\beta^k(g)=g.
	\]
	The condition $N\beta=0$ implies $\phi_\beta^N=\id_R$. Define the right
	action on the PBW basis of $T$ by
	\begin{equation}\label{eq:right-action-existence}
		u\triangleleft G^iX^j=
		\begin{cases}
			\phi_\beta^i(u),&j=0,\\
			0,&j>0.
		\end{cases}
	\end{equation}
	For a right module the operator attached to a product $ab$ is
	$\rho_b\circ\rho_a$. Here $\rho_G=\phi_\beta$ and $\rho_X=0$, so
	\[
	\rho_{G^N}=\phi_\beta^N=\id_R,
	\qquad \rho_{X^n}=0,
	\qquad \rho_{GX}=\rho_X\rho_G=0
	=\xi\rho_G\rho_X=\xi\rho_{XG}.
	\]
	Hence the relations $G^N=1$, $X^n=0$, and $GX=\xi XG$ are respected, and
	\eqref{eq:right-action-existence} is a well-defined right $T$-module
	structure on $R$.
	
	It remains to verify the coalgebra and matched-pair axioms globally.
	
	\subsection{The module-coalgebra conditions}
	
	For the left action, $g$ acts as the identity coalgebra map. The operator $D$
	is a coderivation. Indeed, every summand of
	$\Delta(G^iX^j)$ in \eqref{eq:taft-coproduct} has total $X$-degree $j$;
	therefore \eqref{eq:D-basis} gives
	\[
	\Delta(D(G^iX^j))
	=(D\otimes\id+\id\otimes D)\Delta(G^iX^j).
	\]
	This is precisely the module-coalgebra condition for the
	$(1,g)$-skew-primitive element $x$, because $g$ acts as the identity. The
	counit identity follows from \eqref{eq:D-basis}. Since the action is a
	representation of the algebra $R$ and $\Delta_R$ is multiplicative, the
	module-coalgebra identity for $g$ and $x$ extends to every product in $R$.
	
	For the right action, the condition for $G$ follows because $\phi_\beta$ is a
	coalgebra automorphism. For $X$, both sides of
	\[
	\Delta(u\triangleleft X)
	=\sum u_{(1)}\triangleleft X_{(1)}
	\otimes u_{(2)}\triangleleft X_{(2)}
	\]
	vanish: $u\triangleleft X=0$, and in
	$\Delta(X)=X\otimes1+G\otimes X$ each term contains a right action by $X$ in
	one tensor factor. If the coalgebra-action identity holds for $a,b\in T$,
	then it also holds for $ab$ by the right module identity
	$\rho_{ab}=\rho_b\rho_a$ and the multiplicativity of $\Delta_T$. Hence the
	identity extends from $G,X$ to all of $T$. The counit condition is checked in
	the same way.
	
	The unit identities follow immediately from
	\eqref{eq:left-action-existence} and \eqref{eq:right-action-existence}.
	
	\subsection{Verification of MP1}
	
	We verify MP1 on the PBW basis. Let
	\[
	a=G^iX^j,
	\qquad
	b=G^kX^\ell.
	\]
	If $j+\ell\ge n$, then $ab=0$. On the right-hand side of MP1, the right
	action is nonzero only for the $s=0$ term of $\Delta(a)$, so the same
	calculation as below expresses that side as a scalar multiple of $ab$;
	hence it also vanishes. We may therefore assume $j+\ell<n$. Then $ab\ne0$,
	and the left-hand side is
	\[
	u\triangleright(ab)=\chi_{j+\ell}(u)ab.
	\]
	In the coproduct of $a$, the factor
	$u_{(2)}\triangleleft a_{(2)}$ is nonzero only for the summand with
	$X$-degree $0$ in $a_{(2)}$, namely the $s=0$ term of
	\eqref{eq:taft-coproduct}. Hence the right-hand side of MP1 equals
	\[
	\sum \chi_j(u_{(1)})
	\chi_\ell(\phi_\beta^i(u_{(2)}))\,ab.
	\]
	Since $\chi_\ell\circ\phi_\beta^i=\chi_\ell$, this becomes
	\[
	(\chi_j*\chi_\ell)(u)\,ab.
	\]
	The convolution of these characters satisfies
	\[
	\chi_j*\chi_\ell=\chi_{j+\ell},
	\]
	as is checked on the algebra generators $g,x$ using
	$\Delta(g)=g\otimes g$ and $\Delta(x)=x\otimes1+g\otimes x$. Therefore MP1
	holds for all $u,a,b$.
	
	\subsection{Verification of MP2}
	
	Let $a=G^iX^j$. If $j>0$, consider a summand of
	$\Delta(a)$ indexed by $s$ in \eqref{eq:taft-coproduct}. For
	\[
	u\triangleleft(v_{(1)}\triangleright a_{(1)})
	\]
	to be nonzero, $a_{(1)}$ must have $X$-degree zero, which forces $s=j$.
	For
	\[
	v_{(2)}\triangleleft a_{(2)}
	\]
	to be nonzero, $a_{(2)}$ must have $X$-degree zero, which forces $s=0$.
	These conditions are incompatible when $j>0$. Thus both sides of MP2 are
	zero.
	
	If $j=0$, then $a=G^i$ and
	$v_{(1)}\triangleright G^i=\varepsilon(v_{(1)})G^i$. Therefore the
	right-hand side of MP2 is
	\[
	\sum
	\phi_\beta^i(u)\,
	\varepsilon(v_{(1)})\phi_\beta^i(v_{(2)})
	=\phi_\beta^i(u)\phi_\beta^i(v)
	=\phi_\beta^i(uv),
	\]
	which is exactly $(uv)\triangleleft G^i$. Hence MP2 holds for all elements.
	
	\subsection{Verification of MP3}
	
	Fix $a=G^iX^j$. In the left-hand side of MP3, the right action
	$u_{(1)}\triangleleft a_{(1)}$ is nonzero only when $a_{(1)}$ has
	$X$-degree zero. In \eqref{eq:taft-coproduct} this is the $s=j$ term. Hence
	\[
	\mathrm{LHS}
	=\sum
	\phi_\beta^{i+j}(u_{(1)})
	\otimes\chi_j(u_{(2)})G^iX^j.
	\]
	On the right-hand side, $u_{(2)}\triangleleft a_{(2)}$ is nonzero only for
	$s=0$. Thus
	\[
	\mathrm{RHS}
	=\sum
	\phi_\beta^i(u_{(2)})
	\otimes\chi_j(u_{(1)})G^iX^j.
	\]
	It remains to prove
	\begin{equation}\label{eq:winding-identity}
		\sum\phi_\beta^j(u_{(1)})\chi_j(u_{(2)})
		=\sum\chi_j(u_{(1)})u_{(2)}
		\qquad(u\in R).
	\end{equation}
	Define the right and left winding maps
	\[
	R_{\chi_j}(u)=\sum u_{(1)}\chi_j(u_{(2)}),
	\qquad
	L_{\chi_j}(u)=\sum\chi_j(u_{(1)})u_{(2)}.
	\]
	Because $\chi_j$ is a character, both maps are algebra maps. Therefore
	$\phi_\beta^jR_{\chi_j}=L_{\chi_j}$ may be checked on $g$ and $x$. It is
	immediate for $g$. Put $c=j\beta$. Then
	\[
	R_{\chi_j}(x)=x+cg,
	\qquad
	L_{\chi_j}(x)=x+c,
	\]
	and
	\[
	\phi_\beta^j(x+cg)
	=x+c(1-g)+cg=x+c.
	\]
	This proves \eqref{eq:winding-identity}, and hence MP3 holds for all
	$u\in R$ and $a\in T$.
	
	We have proved the sufficiency statement.
	
	\begin{pro}\label{prop:existence}
		For every $\beta\in\F_p$ satisfying $N\beta=0$, the actions
		\eqref{eq:left-action-existence} and \eqref{eq:right-action-existence} define
		a matched pair $(T,R,\triangleright,\triangleleft)$. These are precisely the
		matched pairs described in Proposition~\ref{prop:second}.
	\end{pro}
	
	\begin{proof}
		All module, module-coalgebra, unit, and compatibility conditions were verified
		above. Proposition~\ref{prop:second} proves that no other matched pairs occur.
	\end{proof}
	
	Combining Propositions~\ref{prop:second} and \ref{prop:existence} proves the
	matched-pair part of Theorem~\ref{thm:main}. In particular, if $p\nmid N$,
	then $N\ne0$ in $\K$, so $N\beta=0$ implies $\beta=0$ and both actions are
	trivial in the standard sense. If $p\mid N$, all $\beta\in\F_p$ occur.
	
	\section{Bicrossed products and isomorphism classes}\label{sec:isom}
	
	For a matched pair with parameter $\beta$, the bicrossed product
	$E_\beta=T\bowtie_\beta R$ has underlying vector space $T\otimes R$ and
	tensor-product coalgebra. Its multiplication gives
	\[
	(1\bowtie g)(G\bowtie1)=G\bowtie g,
	\qquad
	(1\bowtie g)(X\bowtie1)=X\bowtie g,
	\]
	\[
	(1\bowtie x)(G\bowtie1)
	=G\bowtie x+\beta G\bowtie(1-g),
	\]
	\[
	(1\bowtie x)(X\bowtie1)
	=X\bowtie x+\beta X\bowtie1.
	\]
	Suppressing $\bowtie$, the algebra $E_\beta$ is generated by $G,X,g,x$ with
	the defining relations of $T$ and $R$ and the cross relations
	\begin{equation}\label{eq:cross-relations}
		gG=Gg,
		\qquad gX=Xg,
		\qquad xG=Gx+\beta G(1-g),
		\qquad xX=Xx+\beta X.
	\end{equation}
	When $\beta=0$, this is the tensor product Hopf algebra $T\otimes R$.
	When $\beta\ne0$ and $p\mid N$, it is a nontrivial bicrossed product.
	
	The coalgebra of every $E_\beta$ is the same tensor-product coalgebra
	$T\otimes R$. We begin with two elementary lemmas.
	
	\begin{lem}\label{lem:grouplikes-product}
		For every $\beta$,
		\[
		\G(E_\beta)=\{G^ig^j\mid0\le i<N,\ 0\le j<p\}.
		\]
	\end{lem}
	\begin{proof}
		This is a standard fact for tensor-product coalgebras; we include the proof
		for completeness.	Let $\pi_T=\id_T\otimes\varepsilon_R$ and $\pi_R=\varepsilon_T\otimes\id_R$.
		These are coalgebra maps. For the tensor-product coalgebra $T\otimes R$, the
		standard projection identity
		\begin{equation}\label{eq:proj-reconstruction}
			(\pi_T\otimes\pi_R)\Delta(w)=w,\qquad w\in T\otimes R
		\end{equation}
		holds. If $z\in E_\beta$ is group-like, then $\Delta(z)=z\otimes z$, so
		\eqref{eq:proj-reconstruction} gives
		\[
		z=(\pi_T\otimes\pi_R)(z\otimes z)=\pi_T(z)\otimes\pi_R(z).
		\]
		Since $\pi_T(z)\in\G(T)$ and $\pi_R(z)\in\G(R)$, we may write
		$\pi_T(z)=G^i$ and $\pi_R(z)=g^j$, hence $z=G^i\otimes g^j$.
		Conversely, every such tensor is group-like.
	\end{proof}
	
	\begin{lem}\label{lem:primitive-product}
		Let $a\in\G(T)$ and $b\in\G(R)$, and set $h=a\otimes b$. If
		$z\in\Pp_{1,h}(T\otimes R)$, then there exist
		$u\in\Pp_{1,a}(T)$ and $v\in\Pp_{1,b}(R)$ such that
		\begin{equation}\label{eq:primitive-decomposition}
			z=u\otimes 1+a\otimes v,
		\end{equation}
		and the pair $(u,v)$ satisfies the compatibility condition
		\begin{equation}\label{eq:mixed-primitive-condition}
			u\otimes(1-b)+(a-1)\otimes v=0
			\quad\text{in }T\otimes R.
		\end{equation}
		Consequently:
		\begin{enumerate}
			\item if $b=1$, then
			$\Pp_{1,a\otimes 1}(T\otimes R)=\Pp_{1,a}(T)\otimes 1$;
			\item if $a=1$, then
			$\Pp_{1,1\otimes b}(T\otimes R)=1\otimes\Pp_{1,b}(R)$;
			\item if $a\ne 1$ and $b\ne 1$, then
			\[
			\Pp_{1,a\otimes b}(T\otimes R)=\K(a\otimes b-1).
			\]
		\end{enumerate}
	\end{lem}
	\begin{proof}
		We use the projections $\pi_T,\pi_R$ and the projection identity
		\eqref{eq:proj-reconstruction} from Lemma~\ref{lem:grouplikes-product}.
		Since $\pi_T,\pi_R$ are bialgebra maps,
		$u:=\pi_T(z)\in\mathcal P_{1,a}(T)$ and $v:=\pi_R(z)\in\mathcal P_{1,b}(R)$.
		Applying \eqref{eq:proj-reconstruction} to $z$ and using
		$\Delta(z)=z\otimes 1+(a\otimes b)\otimes z$ immediately gives
		\begin{equation}\label{eq:decomp}
			z=u\otimes 1_R+a\otimes v.
		\end{equation}
		Substituting \eqref{eq:decomp} back into the skew-primitive equation, and
		comparing terms after canceling the common outer summands, yields the
		compatibility condition
		\begin{equation}\label{eq:mixed}
			u\otimes(1_R-b)+(a-1_T)\otimes v=0.
		\end{equation}
		
		Now we prove the three consequences.
		\begin{enumerate}
			\item If $b=1_R$, then \eqref{eq:mixed} gives
			$(a-1_T)\otimes v=0$, hence $v=0$ and $z=u\otimes1_R$.
			The converse is clear.
			\item If $a=1_T$, the same argument shows $u=0$ and
			$z=1_T\otimes v$.
			\item Assume $a\ne1_T$ and $b\ne1_R$. Projecting \eqref{eq:mixed}
			onto $(T/\K(a-1_T))\otimes R$ gives $u\in\K(a-1_T)$, and
			projecting onto $T\otimes(R/\K(1_R-b))$ gives $v\in\K(1_R-b)$.
			Writing $u=\lambda(a-1_T)$ and $v=\mu(1_R-b)$ and substituting
			into \eqref{eq:mixed} yields $\lambda+\mu=0$. Hence
			\[
			z=\lambda(a-1_T)\otimes1_R+a\otimes(-\lambda)(1_R-b)
			=\lambda(a\otimes b-1).
			\]
			The converse is immediate.
		\end{enumerate}
	\end{proof}

	\begin{pro}\label{prop:isom}
		Let $\beta,\beta'\in\F_p$ satisfy $N\beta=N\beta'=0$. Then
		\[
		E_\beta\cong E_{\beta'}
		\quad\text{as Hopf algebras}
		\quad\Longleftrightarrow\quad
		\beta=\beta'.
		\]
		Consequently, if $p\mid N$, the $p$ parameters in $\F_p$ give exactly $p$
		pairwise nonisomorphic bicrossed products.
	\end{pro}
	
	\begin{proof}
		The converse is immediate: if $\beta=\beta'$, the identity on the generators
		is a Hopf algebra isomorphism. We prove necessity.
		
		Let
		\[
		\Phi:E_\beta\longrightarrow E_{\beta'}
		\]
		be a Hopf algebra isomorphism. It preserves group-like elements, their
		orders, and the dimensions of all skew-primitive spaces.
		
		We first prove that $\Phi$ fixes $g$ and $G$. By
		Proposition~\ref{prop:primitive-T}, Lemma~\ref{lem:primitive-R}, and
		Lemma~\ref{lem:primitive-product}, the only group-like elements
		$h\in E_{\beta'}$ for which
		$\dim\Pp_{1,h}(E_{\beta'})=2$ are
		\[
		h=G
		\qquad\text{and}\qquad
		h=g.
		\]
		Indeed, for $G^a$ with $a\ne1$ and for $g^b$ with $b\ne1$ the corresponding
		space has dimension one (or zero at the identity), and every mixed group-like
		$G^ag^b$ with $a,b\ne0$ has a one-dimensional skew-primitive space.
		
		Now $\dim\Pp_{1,g}(E_\beta)=2$, so $\Phi(g)$ is either $g$ or $G$.
		However, $g$ has order $p$, whereas $G$ has order $N$. These orders are
		unequal: if $p\nmid N$ this is immediate, while if $p\mid N$, the assumptions
		$n\mid N$, $(p,n)=1$, and $n\ge2$ imply $pn\mid N$, hence $N>p$. Therefore
		\[
		\Phi(g)=g.
		\]
		The same argument applied to $G$ gives
		\[
		\Phi(G)=G.
		\]
		In particular, when $p\mid N$, the element $G^{N/p}$ cannot be the image of
		$g$: it has order $p$, but
		\[
		\dim\Pp_{1,G^{N/p}}(E_{\beta'})=1,
		\]
		because $N/p\ne1$.
		
		Since $\Phi(g)=g$, Lemma~\ref{lem:primitive-product} and
		Lemma~\ref{lem:primitive-R} give
		\[
		\Phi(x)\in\Pp_{1,g}(E_{\beta'})
		=\K(g-1)\oplus\K x.
		\]
		Write
		\[
		\Phi(x)=a(g-1)+bx.
		\]
		Apply $\Phi$ to the relation
		\[
		gx-xg=g(1-g).
		\]
		Because $g$ commutes with $g-1$, the target relation gives
		\[
		b\,g(1-g)=g(1-g).
		\]
		As $g(1-g)\ne0$, we obtain
		\begin{equation}\label{eq:Phi-x}
			b=1,
			\qquad
			\Phi(x)=x+a(g-1).
		\end{equation}
		
		Similarly, $\Phi(G)=G$ implies
		\[
		\Phi(X)\in\Pp_{1,G}(E_{\beta'})
		=\K(G-1)\oplus\K X.
		\]
		Write $\Phi(X)=c(G-1)+dX$. Applying $\Phi$ to $GX=\xi XG$ gives
		\[
		c(1-\xi)G(G-1)=0,
		\]
		so $c=0$. Since $\Phi$ is injective, $d\ne0$, and hence
		\begin{equation}\label{eq:Phi-X}
			\Phi(X)=dX,
			\qquad d\in\K^\times.
		\end{equation}
		
		Finally, apply $\Phi$ to the cross relation
		\[
		xX-Xx=\beta X.
		\]
		Using \eqref{eq:Phi-x}, \eqref{eq:Phi-X}, and the fact that $g$ commutes with
		$X$, the left-hand side in the target is
		\[
		[x+a(g-1),dX]
		=d[x,X]
		=d\beta'X
		=\beta'\Phi(X).
		\]
		On the other hand,
		\[
		\Phi(\beta X)=\beta\Phi(X).
		\]
		Since $\Phi(X)\ne0$, we conclude
		\[
		\beta'=\beta.
		\]
		This proves the proposition.
	\end{proof}
	
	\begin{proof}[Proof of Theorem~\ref{thm:main}]
		The necessity part is Proposition~\ref{prop:second}. The sufficiency
		part is Proposition~\ref{prop:existence}. The isomorphism classification
		is Proposition~\ref{prop:isom}. Combining these three results proves
		Theorem~\ref{thm:main}.
	\end{proof}

	\begin{rmk}[Relation with liftings of  quantum-plane]\label{rmk:quantum-plane}
		Let $G=\Z_N\times\Z_p$, and let $g_1$ and $g_2$ be generators of the
		first and second factors, respectively. Take the YD-datum
		\[
		\mathcal D=(G,g_1,g_2,\chi_1,\epsilon),
		\]
		where $\chi_1(g_1)=\xi$ with $\xi$ a primitive $n$-th root of unity.
		Let $f:G\to\K$ be an additive map, that is,
		\[
		f(hk)=f(h)+f(k)\qquad (h,k\in G),
		\]
		satisfying
		\[
		f(g_1)=-\beta,\qquad f(g_2)=1.
		\]
		The existence of such an $f$ is equivalent to $N\beta=0$, which is
		exactly the parameter restriction obtained in
		Proposition~\ref{prop:second}. Then the algebra
		$\cH^3(\mathcal D,f,0)$ of \cite{Xpq2} is generated by
		$g_1,x,g_2,y$ with defining relations
		\[
		g_1^N=1,\quad x^n=0,\quad g_1x=\xi xg_1,
		\]
		\[
		g_2^p=1,\quad y^p=y,\quad g_2y=yg_2+g_2(1-g_2),
		\]
		\[
		g_2g_1=g_1g_2,\qquad
		g_2x=xg_2,\qquad
		yg_1=g_1y+\beta g_1(1-g_2),\qquad
		yx=xy+\beta x.
		\]
		These are precisely the defining relations of the bicrossed product
		$T_{N,n,\xi}\bowtie_\beta R$.
	\end{rmk}

	\section*{Acknowledgements}
	
	This work was partially supported by the National Natural Science Foundation
	of China (Grant No. 12401041). The author is grateful to the authors of the references for their work, which motivated the present paper.

	\section*{Disclosure statement}
	No potential conflict of interest was reported by the author.

\end{document}